\documentclass[11pt,reqno]{amsart}
\usepackage[margin=1.1in]{geometry}
\usepackage{amsthm}
\usepackage{xcolor}
\usepackage{hyperref}
\usepackage[T1]{fontenc}
\usepackage{anyfontsize}
\usepackage{amsmath,amssymb,mathtools}
\usepackage[expansion=false]{microtype}
\hypersetup{
 colorlinks=true,linkcolor=blue!45!black,citecolor=blue!45!black,urlcolor=blue!45!black,
 pdftitle={Sobolev mixing constants and compactness in rational moduli space},
 pdfauthor={Alastair Fletcher and Ilya Krishtal},
 pdfsubject={Version 10, 13 September 2026; ETDS submission; critical Sobolev mixing and compactness in rational moduli space},
 pdfkeywords={uniformly quasiregular dynamics, Sobolev space, mixing, transfer operator, rational moduli space}
}
\theoremstyle{plain}
\newtheorem{theorem}{Theorem}[section]
\newtheorem{proposition}[theorem]{Proposition}
\newtheorem{lemma}[theorem]{Lemma}
\newtheorem{corollary}[theorem]{Corollary}
\theoremstyle{definition}

\theoremstyle{remark}
\newtheorem{remark}[theorem]{Remark}

\newcommand{\R}{\mathbb R}
\newcommand{\C}{\mathbb C}
\newcommand{\T}{\mathbb T}

\newcommand{\E}{\mathbb E}
\newcommand{\Lop}{\mathcal L}
\newcommand{\Hcal}{\mathcal H}
\newcommand{\Bcal}{\mathcal B}
\newcommand{\Dcal}{\mathcal D}
\newcommand{\Tr}{\operatorname{Tr}_{\mu}}
\newcommand{\norm}[1]{\left\lVert #1\right\rVert}
\newcommand{\abs}[1]{\left\lvert #1\right\rvert}
\newcommand{\ip}[2]{\left\langle #1,#2\right\rangle}
\newcommand{\dd}{\mathrm d}
\DeclareMathOperator{\vol}{vol}
\DeclareMathOperator{\supp}{supp}
\DeclareMathOperator{\spec}{\sigma}

\DeclareMathOperator{\dist}{dist}

\numberwithin{equation}{section}

\begin{document}

\title{Sobolev mixing constants and compactness in rational moduli space}

\author[A. Fletcher]{Alastair Fletcher}
\address{Department of Mathematical Sciences, Northern Illinois University,
DeKalb, IL 60115, USA}
\email{afletcher@niu.edu}

\author[I. Krishtal]{Ilya Krishtal}
\address{Department of Mathematical Sciences, Northern Illinois University,
DeKalb, IL 60115, USA}
\email{ikrishtal@niu.edu}


\begin{abstract}
We study uniformity of Sobolev mixing estimates for rational maps and uniformly quasiregular mappings. For rational maps of fixed degree, the optimal centered Sobolev trace constant $A_f$ is a continuous proper function on M\"obius moduli space. Uniform bounds on $A_f$ therefore characterize relative compactness in moduli, and a minimizing class exists in every degree. In the setting of uniformly quasiregular endomorphisms of degree $d\ge2$ on closed $n$-manifolds with an invariant conformal structure, the $k$th centered transfer operator from critical Sobolev energy into $L^1$ of the equilibrium measure has norm $A_f d^{-k/n}$. On the mean-zero Sobolev space, the spectrum and Fredholm essential spectrum are the closed disk of radius $d^{-1/n}$, with infinite-dimensional eigenspaces throughout its interior. The proofs use the energy scaling of pullback, a bounded equilibrium trace, and an obstruction from atoms of intermediate mass. Combined with conformal barycenter normalization and DeMarco--Faber's degeneration theorem, this obstruction gives the moduli compactness criterion. Explicit families illustrate the distinction between degeneration and concentration caused by changes of coordinates.
\end{abstract}

\maketitle

\section{Introduction}

For rational maps of degree $d\ge2$, correlations of continuous $W^{1,2}$ observables against bounded observables admit an exponential bound with factor $d^{-k/2}$, with respect to the equilibrium measure \cite{DNS07,DS06}. We study the dependence of the optimal prefactor on the map.

For a family of fixed degree, the common exponential factor leaves a uniformity question: when can the prefactors be bounded independently of the map?

We show that a single quantity decides this. An optimal trace constant $A_f$, defined by an inequality independent of the iterate, determines the norm of every centered transfer iterate. It is unchanged by conformal changes of coordinates and defines a continuous proper function on the moduli space of degree $d$ maps. In particular, it tends to infinity along any sequence escaping every compact subset of moduli space. Uniform critical Sobolev mixing estimates across a family are therefore equivalent to relative compactness of the family in moduli, and a minimizing class exists in every degree.

The trace inequality compares the oscillation of a function, measured against the equilibrium measure on the Julia set, with its Dirichlet energy, measured on the whole sphere. Since the equilibrium measure may be singular with respect to spherical area, a bounded trace is needed to extend this comparison to Sobolev classes. This is also where degeneration shows up. At the critical Sobolev exponent a point has zero capacity. If a weak limit retains positive mass both at a point and away from it, suitable cutoffs have vanishing energy but positive centered oscillation, forcing the optimal constant to diverge. The same trace and capacity arguments apply in higher dimensions, and we develop them for uniformly quasiregular mappings.

\subsection{Uniformity and rational moduli}

Throughout, $S^2$ carries the spherical metric, and $m$ denotes normalized spherical area. The normalization of $m$ depends on this choice of metric, not only on its conformal class. Let $f$ be a rational map of degree $d\ge2$ on $S^2$, and let $\mu=\mu_f$ be its equilibrium measure. Set
\begin{equation}\label{eq:rational-energy}
 E(\varphi)=\left(\int_{S^2}|\dd\varphi|^2\,\dd m\right)^{1/2},
 \qquad \Dcal=C(S^2)\cap W^{1,2}(S^2).
\end{equation}
The optimal centered trace constant is
\begin{equation}\label{eq:optimal-trace}
 A_f=\sup_{\substack{\varphi\in\Dcal\\E(\varphi)>0}}
 \frac{\norm{\varphi-\mu(\varphi)}_{L^1(\mu)}}{E(\varphi)},
 \qquad \mu(\varphi)=\int\varphi\,\dd\mu.
\end{equation}
Thus $A_f$ compares the mean oscillation of an observable $\varphi\in\Dcal$ with its Dirichlet energy. Theorem~\ref{thm:main} below gives its finiteness and its exact interpretation as a mixing prefactor.

Let $\mathrm{Rat}_d$ be the space of rational maps of degree $d$, with its coefficient topology, and let
\[
 \mathcal M_d=\mathrm{Rat}_d/\mathrm{PSL}_2(\C)
\]
be the space of M\"obius conjugacy classes, equipped with the quotient topology. The conjugation action is proper for $d\ge2$ \cite[Lemma 4.3]{BFN22}, hence $\mathcal M_d$ is Hausdorff and locally compact. Conformal invariance of the Dirichlet energy makes $A_f$ constant on these conjugacy classes.

\begin{theorem}\label{thm:moduli-properness}
For every $d\ge2$, the function $f\mapsto A_f$ descends to a continuous function on $\mathcal M_d$ whose sublevel sets
\begin{equation}\label{eq:moduli-sublevel}
 \mathcal K_{d,C}=\{[f]\in\mathcal M_d:A_f\le C\}
\end{equation}
are compact for every $C>0$. In particular, for any family $\mathcal F\subset\mathrm{Rat}_d$,
\begin{equation}\label{eq:moduli-uniformity}
 \sup_{f\in\mathcal F}A_f<\infty
 \quad\Longleftrightarrow\quad
 \{[f]:f\in\mathcal F\}\text{ is relatively compact in }\mathcal M_d.
\end{equation}
Consequently, $A_f$ is a proper function on $\mathcal M_d$.
\end{theorem}

In particular, $A_{f_j}\to\infty$ whenever $[f_j]$ escapes every compact subset of $\mathcal M_d$. The proof combines continuity of $A_f$ with the degeneration theory of DeMarco \cite{DeM05} and DeMarco--Faber \cite{DF14}, which shows that a boundary limit of equilibrium measures is purely atomic. A single point mass would be harmless, since centering removes it; conformal barycenter normalization \cite{DE86} excludes that case and leaves an atom of mass strictly between zero and one. Such an atom forces the centered trace constants to diverge, because a point has zero critical Sobolev capacity. A minimizing class exists in every degree; see Corollary~\ref{cor:moduli-minimum}.

\subsection{Critical energy and the transfer operator}

The analytic argument extends to uniformly quasiregular mappings, whose iterates satisfy a common bound $K$ on the maximal dilatation. Let $M$ be a closed, connected, oriented Riemannian manifold of dimension $n\ge2$, and let $f\colon M\to M$ be a uniformly quasiregular endomorphism of degree $d\ge2$. We now write $m=\vol_M/\vol_M(M)$ and $\mu=\mu_f$. Okuyama and Pankka \cite{OP14} constructed this non-atomic equilibrium measure, proved the balance relation $f^*\mu=d\mu$, and established strong mixing. The associated normalized transfer operator is
\begin{equation}\label{eq:intro-transfer}
 \Lop\varphi(y)=\frac1d\sum_{x\in f^{-1}(y)}i(x,f)\varphi(x),
\end{equation}
where $i(x,f)$ denotes the local index. We write $U\varphi=\varphi\circ f$ for pullback.

Fix an $f$-invariant measurable conformal structure $G$, whose existence follows from Iwaniec and Martin \cite[Theorem 5.1]{IM96}; see also \cite[Section 21.5]{IM01}. Its formulation on a closed manifold is recalled in \cite[Section 4]{OP14}. Denote the induced covector norm by $|\cdot|_G$ and extend the preceding notation by setting
\begin{equation}\label{eq:energy-intro}
 E(\varphi)=\left(\int_M|\dd\varphi|_G^n\,\dd m\right)^{1/n},
 \qquad \Dcal=C(M)\cap W^{1,n}(M),
 \qquad r=d^{-1/n}.
\end{equation}
The space $W^{1,n}(M)$ uses the smooth background metric. Bounded ellipticity of $G$ makes $E$ equivalent to the usual gradient seminorm, so $E$ vanishes precisely on constants. We define $A_f$ by \eqref{eq:optimal-trace} with this energy, test space, and equilibrium measure. Its dependence on $G$ and the background metric is understood. For rational maps with the standard conformal structure, this recovers \eqref{eq:rational-energy}.

The usefulness of the critical exponent is expressed by three relations, proved in Lemma~\ref{lem:energy}:
\begin{equation}\label{eq:intro-operator-mechanism}
 \Lop U=I,\qquad E(U\varphi)=r^{-1}E(\varphi),\qquad
 E(\Lop\varphi)\le rE(\varphi),\qquad\varphi\in\Dcal.
\end{equation}
The first is immediate from the index sum $\sum_{f(x)=y}i(x,f)=d$; the other two use invariance of $G$ and the change of variables for quasiregular maps. In the rational case, pullback by a map of degree $d$ multiplies the Dirichlet integral by $d$. The transfer contraction follows from the Cauchy--Schwarz calculation in the proof of \cite[Proposition 4.2]{DNS07}, with the perturbation parameter set to zero.
The energy contraction gives an upper estimate, while the right inverse supplies a matching lower bound. A bounded trace into the equilibrium measure connects these energy identities with mixing.

\begin{theorem}\label{thm:main}
There is a bounded trace map $\Tr\colon W^{1,n}(M)\to L^1(\mu)$ which agrees with restriction on $\Dcal$. The constant $A_f$ is finite and positive. For every integer $k\ge0$,
\begin{equation}\label{eq:sharp-mixing}
 \sup_{\substack{\varphi\in\Dcal\\E(\varphi)>0}}
 \frac{\norm{\Lop^k\varphi-\mu(\varphi)}_{L^1(\mu)}}{E(\varphi)}
 =A_f d^{-k/n}.
\end{equation}
\end{theorem}

For $\varphi\in\Dcal$ and $\psi\in L^\infty(\mu)$, transfer duality therefore gives
\begin{equation}\label{eq:intro-correlation}
 \abs{\int_M\varphi(\psi\circ f^k)\,\dd\mu-\mu(\varphi)\mu(\psi)}
 \le A_f d^{-k/n}E(\varphi)\norm{\psi}_{L^\infty(\mu)}.
\end{equation}
The factor $A_f d^{-k/n}$ is optimal when both observables vary over the indicated unit balls; the maximizing or approximating observables may depend on $k$. Because the constant in \eqref{eq:sharp-mixing} is exact, a family admits uniform critical Sobolev mixing estimates precisely when its trace constants are uniformly bounded. In the rational case, Theorem~\ref{thm:moduli-properness} identifies that condition with relative compactness in moduli. A quantitative bound for $A_f$ in terms of the normalized Jacobian discrepancy is given in Proposition~\ref{prop:prefactor-bound}.

The right inverse also determines the operator norm on the centered Sobolev space. Non-injectivity of $f$ supplies an infinite-dimensional transfer kernel; iterated pullbacks of its elements yield the eigenvectors for the next result. Here $W^{1,n}(M)$ is complexified, and $\mu(v)=\int_M\Tr v\,\dd\mu$, where $\Tr$ is the trace map of Theorem~\ref{thm:main}. A Sobolev class does not determine the values of an arbitrary representative on a set of volume zero. The trace assigns a canonical $L^1(\mu)$ class by completion from smooth functions and recovers \eqref{eq:optimal-trace} on $\Dcal$.

\begin{theorem}\label{thm:spectrum-intro}
The transfer operator extends boundedly to $W^{1,n}(M)$. On
\[
 \Hcal_0=\{v\in W^{1,n}(M):\mu(v)=0\},
\]
the energy $E$ is a complete norm, and
\begin{equation}\label{eq:spectrum-intro}
 \norm{\Lop^k|_{\Hcal_0}}=r^k,
 \qquad
 \spec(\Lop|_{\Hcal_0})=\{z\in\C:|z|\le r\}.
\end{equation}
Every $|z|<r$ is an eigenvalue of $\Lop|_{\Hcal_0}$ with infinite-dimensional eigenspace. Its Fredholm essential spectrum is also the closed disk $\{|z|\le r\}$. Consequently,
\[
 \spec(\Lop\colon W^{1,n}(M)\to W^{1,n}(M))
 =\{1\}\cup\{z\in\C:|z|\le r\},
\]
and $1$ is a simple isolated eigenvalue.
\end{theorem}

Sobolev methods for weighted transfer operators in rational dynamics occur already in Smirnov \cite[Chapter 2, Section 3]{Smi96}. The correlation rate in \eqref{eq:intro-correlation} for $n=2$ is contained in Dinh and Sibony \cite[Theorem 6.1]{DS06}. Dinh, Nguy\^en and Sibony \cite[Sections 2 and 3]{DNS07} establish equilibrium traces, an equivalent centered Sobolev norm, and a spectral gap on $W^{1,2}$. Theorems~\ref{thm:main} and \ref{thm:spectrum-intro} identify the optimal norm at every iterate, the full and Fredholm essential spectra, and the multiplicity of each interior eigenvalue. These statements concern the full Sobolev space. Its trace can have a nontrivial kernel, as explained in Section~\ref{sec:trace}.

\subsection{Examples and the proof of compactness}

Conformal changes of coordinates explain the role of moduli space. The maps $h_a(z)=az$ conjugate $z^2$ to $z^2/a$. As $a\downarrow0$, their equilibrium measures concentrate at the origin, whereas $A_{z^2/a}=A_{z^2}$ by conformal invariance. The following family exhibits degeneration with an unbounded optimal constant.

\begin{theorem}\label{thm:nonuniform-intro}
On $S^2$, consider the quadratic rational maps
\begin{equation}\label{eq:degenerate-family}
 f_t(z)=\frac{z(z-1)}{2(z-1)+t},\qquad 0<t<1.
\end{equation}
With the standard conformal structure, their optimal constants satisfy
\[
 A_{f_t}\longrightarrow\infty\qquad(t\downarrow0).
\]
In particular, no bound for the prefactor in \eqref{eq:intro-correlation} can depend only on the dimension $n$, the maximal dilatation $K$, and the degree $d$.
\end{theorem}

The inverse-image equation gives a direct calculation of the limiting equilibrium measure. An atom of intermediate mass retains positive centered oscillation on cutoffs whose critical energies tend to zero. This proves the divergence in Theorem~\ref{thm:nonuniform-intro}. For general rational families, conformal barycenter normalization and DeMarco--Faber's theorem provide the boundary measures to which the same argument applies. The trace obstruction itself holds on closed manifolds in every dimension $n\ge2$ for structures with a common ellipticity bound.

Sections~\ref{sec:energy} and \ref{sec:trace} establish the energy identities and the equilibrium trace, including the quantitative and variational descriptions of its bound. Section~\ref{sec:spectrum} proves the exact norm and spectrum results. Section~\ref{sec:statistics} derives H\"older mixing estimates, exactness, mixing of all orders, and a central limit theorem. Section~\ref{sec:families} develops continuity, quasiconformal conjugacy, and the atomic obstruction, then proves Theorem~\ref{thm:moduli-properness}. Section~\ref{sec:examples} gives the examples and proves Theorem~\ref{thm:nonuniform-intro}.

\section{The critical energy and the transfer operator}\label{sec:energy}

The invariant metric makes the critical energy scale exactly under pullback. We establish this identity together with the transfer contraction, using the global Sobolev mapping properties to pass across the branch set. All Sobolev norms use normalized background volume $m$. For complex-valued functions, $|\dd v|_G$ is the Hilbert norm of the complexified covector; the real and imaginary parts are treated together.

\subsection{The invariant metric}

For a quasiregular map $g$, its outer and inner dilatations $K_O(g)$ and $K_I(g)$ are the least constants for which
\[
 \norm{Dg}^n\le K_O(g)J_g,
 \qquad J_g\le K_I(g)\ell(Dg)^n
 \quad\text{almost everywhere},
\]
where $\ell(Dg(x))=\inf_{|v|=1}|Dg(x)v|$ is the least stretch in the background metric. The maximal dilatation is $K(g)=\max\{K_O(g),K_I(g)\}$. A map is uniformly $K$-quasiregular if every iterate is quasiregular and $K(f^k)\le K$ for all $k\ge1$. In particular, each iterate belongs to $W^{1,n}_{\mathrm{loc}}$ and satisfies $\norm{Df^k}^n\le KJ_{f^k}$ almost everywhere.

The invariant structure is a measurable positive definite symmetric endomorphism $G$ of $TM$, with determinant one and bounded ellipticity, such that
\begin{equation}\label{eq:beltrami}
 (Df^k(x))^{\mathsf T}G(f^k(x))Df^k(x)
 =J_{f^k}(x)^{2/n}G(x)
\end{equation}
almost everywhere, for every $k\ge1$. We use the metric $\ip{v}{w}_G=\ip{Gv}{w}$ on vectors and its dual metric on covectors. The matrix defining the dual covector norm is $G^{-1}$. There is $\Lambda\ge1$ such that
\begin{equation}\label{eq:ellipticity}
 \Lambda^{-1}|\xi|\le |\xi|_G\le\Lambda|\xi|
\end{equation}
for covectors almost everywhere. Since $\det G=1$, the induced volume form equals the background volume form.

We use the standard change-of-variables and Sobolev composition properties of quasiregular maps; see \cite{HKM93, Rickman93}. In particular, the branch values have volume zero, and inverse branches off the branch values are quasiconformal. The pushforward of a continuous $W^{1,n}$ function is again continuous and in $W^{1,n}$; the relevant pushforward result is \cite[Lemmas 14.30 and 14.31]{HKM93}, used in \cite[Section 2]{OP14}. This global mapping property justifies the weak derivative across the branch values in the inverse-branch calculation below.

For every $1\le p<\infty$, the Sobolev--Poincar\'e inequality and \eqref{eq:ellipticity} give a finite constant $S_p$ such that
\begin{equation}\label{eq:sobolev-poincare}
 \norm{v-m(v)}_{L^p(m)}\le S_p E(v),\qquad v\in W^{1,n}(M);
\end{equation}
see, for example, \cite[Chapter 2]{Hebey00}. Smooth functions are dense in $W^{1,n}(M)$. If a Sobolev function is continuous, smoothing in a finite atlas gives approximation both uniformly and in $W^{1,n}$.

\subsection{Contraction and a right inverse}

Let $U\varphi=\varphi\circ f$. The averaging in \eqref{eq:intro-transfer} gives $\Lop1=1$, positivity, and contraction on $C(M)$ in the uniform norm. Local indices also give $\Lop^k=(f^k)_*/d^k$ and $\Lop U=I$ on continuous functions.

\begin{lemma}\label{lem:energy}
For every $\varphi\in\Dcal$, the functions $\Lop\varphi$ and $U\varphi$ belong to $\Dcal$, and
\begin{equation}\label{eq:energy-pair}
 E(\Lop\varphi)\le rE(\varphi),
 \qquad E(U\varphi)=r^{-1}E(\varphi),
 \qquad \Lop U\varphi=\varphi.
\end{equation}
Consequently, $E(\Lop^k\varphi)\le r^kE(\varphi)$ and $E(U^k\varphi)=r^{-k}E(\varphi)$ for every $k\ge0$.
\end{lemma}

\begin{proof}
Away from the branch values, choose an evenly covered neighborhood with inverse branches $g_1,\ldots,g_d$. Inverting \eqref{eq:beltrami} shows that
\[
 (Dg_j(y))^{\mathsf T}G(g_j(y))Dg_j(y)
 =J_{g_j}(y)^{2/n}G(y).
\]
The dual identity and the Sobolev chain rule therefore give
\[
 |\dd(\varphi\circ g_j)|_G(y)
 =|\dd\varphi|_G(g_j(y))J_{g_j}(y)^{1/n}
\]
almost everywhere. Jensen's inequality yields
\begin{equation}\label{eq:pointwise-transfer}
 |\dd\Lop\varphi|_G^n(y)
 \le\frac1d\sum_{j=1}^d
 |\dd\varphi|_G^n(g_j(y))J_{g_j}(y).
\end{equation}
To integrate this local identity, take a countable cover by evenly covered neighborhoods and disjointify it into measurable subsets. Apply the inverse-branch change of variables on each subset and sum. Each preimage is counted once, apart from a volume-null set. Thus
\[
 \int_M|\dd\Lop\varphi|_G^n\,\dd m
 \le\frac1d\int_M|\dd\varphi|_G^n\,\dd m.
\]
The global Sobolev property recalled above justifies interpreting the derivative in this estimate as the weak derivative on all of $M$.

For the pullback, \eqref{eq:beltrami} gives
\[
 |\dd(\varphi\circ f)|_G^n
 =(|\dd\varphi|_G^n\circ f)J_f
\]
almost everywhere. The degree formula then gives $E(U\varphi)^n=dE(\varphi)^n$. The identity $\Lop U=I$ follows because the sum of the local indices in each fiber is $d$. Iteration proves the final assertions.
\end{proof}

The Jacobian also controls the constant part of a transferred function. Write
\begin{equation}\label{eq:rho}
 \rho_f=\frac{J_f}{d}-1,
 \qquad \eta_f(v)=\int_Mv\rho_f\,\dd m.
\end{equation}
The degree formula gives $m(\rho_f)=0$, and change of variables gives
\begin{equation}\label{eq:mean-transfer}
 m(\Lop v)=m(v)+\eta_f(v),\qquad v\in\Dcal.
\end{equation}
Define the dual critical Sobolev norm of this discrepancy by
\begin{equation}\label{eq:def-b}
 b_f=\sup_{\substack{v\in W^{1,n}(M)\\E(v)\le1}}|\eta_f(v)|.
\end{equation}
Higher integrability of quasiregular mappings gives $J_f\in L^q(m)$ for some $q>1$, as recorded in the proof of \cite[Proposition 3.1]{OP14}. For every such $q$, with $q'=q/(q-1)$,
\begin{equation}\label{eq:dual-bound}
 |\eta_f(v)|
 \le\norm{\rho_f}_{L^q(m)}\norm{v-m(v)}_{L^{q'}(m)}
 \le S_{q'}\norm{\rho_f}_{L^q(m)}E(v).
\end{equation}
In particular, $b_f\le S_{q'}\norm{\rho_f}_{L^q(m)}<\infty$. The energy contraction and this control of the mean give the Sobolev extension of the transfer operator.

\begin{proposition}\label{prop:L-sobolev}
There is a unique bounded extension of $\Lop$ from smooth functions to $W^{1,n}(M)$. It agrees with \eqref{eq:intro-transfer} on $\Dcal$, and
\begin{equation}\label{eq:W-bound}
 E(\Lop v)\le rE(v),\qquad
 |m(\Lop v)|\le |m(v)|+b_fE(v).
\end{equation}
\end{proposition}

\begin{proof}
The expression $|m(v)|+E(v)$ is an equivalent norm on $W^{1,n}(M)$. Lemma~\ref{lem:energy} and \eqref{eq:mean-transfer}--\eqref{eq:dual-bound} prove boundedness in this norm on smooth functions. Completion gives the extension and \eqref{eq:W-bound}. For $v\in\Dcal$, choose smooth $v_j\to v$ both uniformly and in $W^{1,n}$. The functions $\Lop v_j$ converge uniformly to the continuous pushforward of $v$, and in $W^{1,n}$ to its Sobolev extension. The two limits agree.
\end{proof}

\section{The equilibrium trace and its constant}\label{sec:trace}

We estimate the normalized pullback limit by telescoping the one-step Jacobian discrepancy. Applying that estimate to absolute values gives the bounded trace into the equilibrium measure and a quantitative bound for $A_f$.

\subsection{A telescoping estimate}

For $k\ge0$, define the probability measure
\[
 m_k=\frac{(f^k)^*m}{d^k};
 \qquad m_k(v)=m(\Lop^kv),\quad v\in C(M).
\]

\begin{lemma}\label{lem:telescoping}
The measures $m_k$ converge weakly to $\mu$. For every $v\in\Dcal$ and $k\ge0$,
\begin{equation}\label{eq:potential-estimate}
 |\mu(v)-m_k(v)|\le\frac{b_f}{1-r}r^kE(v).
\end{equation}
In particular, with $P_f=b_f/(1-r)$,
\begin{equation}\label{eq:potential-zero}
 |\mu(v)-m(v)|\le P_fE(v).
\end{equation}
\end{lemma}

\begin{proof}
Equations \eqref{eq:mean-transfer} and \eqref{eq:energy-pair} give
\[
 m_{j+1}(v)-m_j(v)=\eta_f(\Lop^jv),
 \qquad |m_{j+1}(v)-m_j(v)|\le b_fr^jE(v).
\]
Thus $m_k(v)$ is Cauchy for every $v\in\Dcal$. Smooth functions are uniformly dense in $C(M)$, and all $m_k$ are probability measures. It follows by uniform approximation that $m_k(v)$ converges for every continuous $v$, to a positive functional of mass one. The corresponding probability measure is the normalized pullback limit, hence the equilibrium measure of \cite{OP14}. Summing the geometric tail proves \eqref{eq:potential-estimate}.
\end{proof}

The construction also gives $\mu(\Lop v)=\mu(v)$ for continuous $v$. Combining this with $\Lop U=I$ gives $\mu(Uv)=\mu(v)$. Moreover, $\Lop((\psi\circ f)\varphi)=\psi\Lop\varphi$ for continuous functions. Consequently,
\begin{equation}\label{eq:duality}
 \int_M(\psi\circ f)\varphi\,\dd\mu
 =\int_M\psi\Lop\varphi\,\dd\mu.
\end{equation}
The balance relation gives the same identities for bounded measurable functions. Positivity and Jensen's inequality extend $\Lop$ to a contraction on $L^p(\mu)$, $1\le p\le\infty$. The map $U$ is an isometry on each of these spaces, and $\Lop=U^*$ on $L^2(\mu)$.

\subsection{Restriction to the equilibrium measure}

The next argument uses the potential estimate on an absolute value. This is the step that converts weak control of the measure into an $L^1$ trace inequality.

\begin{lemma}\label{lem:trace}
For every $v\in\Dcal$,
\begin{align}
 \norm{v-m(v)}_{L^1(\mu)}&\le (S_1+P_f)E(v),\label{eq:uncentered-trace}\\
 \norm{v-\mu(v)}_{L^1(\mu)}&\le 2(S_1+P_f)E(v).\label{eq:centered-trace}
\end{align}
\end{lemma}

\begin{proof}
Let $c=m(v)$ and $w=|v-c|$. The Sobolev chain rule gives $w\in\Dcal$ and $E(w)\le E(v)$, including for complex-valued $v$. Applying \eqref{eq:potential-zero} to $w$ gives
\[
 \int_Mw\,\dd\mu
 \le\int_Mw\,\dd m+P_fE(w)
 \le(S_1+P_f)E(v).
\]
This is \eqref{eq:uncentered-trace}. Finally,
\[
 \norm{v-\mu(v)}_{L^1(\mu)}
 \le\norm{v-c}_{L^1(\mu)}+|\mu(v)-c|
 \le2\norm{v-c}_{L^1(\mu)},
\]
and the lemma is proved.
\end{proof}

The centered estimate gives a bound for the optimal constant in terms of the first normalized pullback of volume.

\begin{proposition}\label{prop:prefactor-bound}
With $b_f$ defined by \eqref{eq:def-b},
\begin{equation}\label{eq:main-constant}
 A_f\le C_f:=2\left(S_1+\frac{b_f}{1-r}\right),
 \qquad
 b_f\le S_{q'}\norm{\frac{J_f}{d}-1}_{L^q(m)}
\end{equation}
for every $q>1$ such that $J_f\in L^q(m)$, where $q'=q/(q-1)$. In particular, $A_f<\infty$.
\end{proposition}

\begin{proof}
Take the supremum in \eqref{eq:centered-trace} and use $P_f=b_f/(1-r)$. The second inequality follows from \eqref{eq:dual-bound}.
\end{proof}

The same estimates extend restriction from continuous observables to Sobolev equivalence classes.

\begin{proposition}\label{prop:trace-extension}
Restriction on smooth functions has a unique bounded extension
\[
 \Tr\colon W^{1,n}(M)\longrightarrow L^1(\mu).
\]
It agrees with restriction on $\Dcal$. With $\mu(v)=\int_M\Tr v\,\dd\mu$, the estimates \eqref{eq:potential-zero} and \eqref{eq:centered-trace} extend to all $W^{1,n}(M)$. Furthermore,
\begin{equation}\label{eq:trace-intertwine}
 \Tr(\Lop v)=\Lop(\Tr v).
\end{equation}
\end{proposition}

\begin{proof}
By \eqref{eq:uncentered-trace},
\[
 \norm{v}_{L^1(\mu)}\le |m(v)|+(S_1+P_f)E(v),\qquad v\in\Dcal.
\]
If smooth $v_j$ converge in $W^{1,n}$, their restrictions are therefore Cauchy in $L^1(\mu)$. This defines the bounded trace and proves its independence of the approximation. Uniform Sobolev smoothing shows agreement on $\Dcal$. All the asserted inequalities pass to the limit.

For smooth $v_j\to v$ in $W^{1,n}$, Proposition~\ref{prop:L-sobolev} gives $\Lop v_j\to\Lop v$ in $W^{1,n}$. Each $\Lop v_j$ is continuous, so its trace is its restriction. The left side of \eqref{eq:trace-intertwine} is the $L^1(\mu)$ limit of these restrictions. Since $\Lop$ is an $L^1(\mu)$ contraction, the same sequence converges to its right side.
\end{proof}

The trace assigns an $L^1(\mu)$ class to each Sobolev class by completion from smooth functions. It can have a nontrivial kernel: for example, a smooth function supported away from $\supp\mu$ has zero trace.

The seminorm $E$ is a complete norm on $\Hcal_0$. Indeed, \eqref{eq:potential-zero} shows that
\begin{equation}\label{eq:norm-equivalence}
 \norm{v}_{\mu,E}:=|\mu(v)|+E(v)
\end{equation}
is equivalent to $|m(v)|+E(v)$ on $W^{1,n}(M)$. The subspace $\Hcal_0$ is closed in that space. The same density argument shows that $A_f$ is the optimal centered trace constant on all of $W^{1,n}(M)$, with the trace interpretation.

\begin{remark}\label{rem:density}
If $\mu=w m$ with $w\in L^p(m)$ for some $p>1$, then
\begin{equation}\label{eq:density-bound}
 A_f\le2S_{p'}\norm{w}_{L^p(m)},\qquad p'=p/(p-1).
\end{equation}
Indeed, H\"older's inequality bounds $\norm{v-m(v)}_{L^1(\mu)}$ by $S_{p'}\norm{w}_pE(v)$, and centering at $\mu(v)$ costs at most a factor of two.

For a manifold that is not a rational homology sphere, Kangasniemi, Okuyama, Pankka and Sahlsten \cite{KOPS21} proved that the topological entropy is $\log d$, and Kangasniemi \cite{Kan21} proved absolute continuity of $\mu$. The proof of \cite[Theorem 1.2]{Kan21} also gives the required higher integrability as follows. Choose $0<k<n$ with nonzero cohomology. The proof constructs a normalized eigenform $\omega$; Proposition 4.6 there gives $\omega\in L^s$ for some $s>n/k$, and Lemma 6.1 identifies
\[
 \mu=|\omega|_f^{n/k}\vol_M.
\]
Here $|\cdot|_f$ is the form norm for the invariant structure used in \cite{Kan21}, equivalent to the background norm. Thus the density relative to $m$ lies in $L^{sk/n}(m)$, with $sk/n>1$. The exponent and density norm can depend on the map. A common exponent and density-norm bound give a uniform trace bound on a family through \eqref{eq:density-bound}.
\end{remark}

\subsection{A variational description}

The constant $b_f$ can also be expressed through a nonlinear Poisson equation. The associated nonlinear flux minimizes the dual norm among all fluxes with the prescribed divergence.

Let $\nabla_Gv$ be the vector field dual to $\dd v$ in the metric $G$, and put $n'=n/(n-1)$. We work with real-valued functions in the following proposition. The value of $b_f$ is unchanged: for a complex test function, rotation by a constant phase and taking the real part recovers the absolute value of $\eta_f(v)$ without increasing its energy.

\begin{proposition}\label{prop:variational}
There is a unique real $u\in W^{1,n}(M)$ with $m(u)=0$ such that
\begin{equation}\label{eq:weak-poisson}
 \int_M|\nabla_Gu|_G^{n-2}
       \ip{\nabla_Gu}{\nabla_Gv}_G\,\dd m
 =\int_M\rho_fv\,\dd m,
 \qquad v\in W^{1,n}(M;\R).
\end{equation}
It satisfies
\begin{equation}\label{eq:dual-exact}
 E(u)^{n-1}=b_f
 =\min_X\norm{X}_{L^{n'}(m;G)},
\end{equation}
where the minimum is over real vector fields $X\in L^{n'}(m;TM)$ satisfying
\begin{equation}\label{eq:flux-constraint}
 \int_M\ip{X}{\nabla_Gv}_G\,\dd m=\int_M\rho_fv\,\dd m
 \quad\text{for all }v\in W^{1,n}(M;\R).
\end{equation}
The minimizing field is $X=|\nabla_Gu|_G^{n-2}\nabla_Gu$.
\end{proposition}

\begin{proof}
On the mean-zero Sobolev space, minimize
\[
 \mathcal F(v)=\frac1n E(v)^n-\eta_f(v).
\]
The estimate $|\eta_f(v)|\le b_fE(v)$ gives coercivity. Reflexivity, weak lower semicontinuity of the energy, and weak continuity of $\eta_f$ give a minimizer. Strict convexity in the gradient gives uniqueness after fixing the mean. The Euler equation is \eqref{eq:weak-poisson}.

Testing with $v=u$ gives $E(u)^n=\eta_f(u)\le b_fE(u)$. On the other hand, H\"older's inequality in \eqref{eq:weak-poisson} gives
\[
 |\eta_f(v)|\le E(u)^{n-1}E(v).
\]
These two inequalities prove $E(u)^{n-1}=b_f$, also when both vanish. Every field satisfying \eqref{eq:flux-constraint} has $b_f\le\norm{X}_{n',G}$. The displayed nonlinear flux has norm $E(u)^{n-1}$ and is admissible by \eqref{eq:weak-poisson}. It attains the minimum.
\end{proof}

The variational formula identifies $b_f$ intrinsically. For estimating the trace constant, the direct bound \eqref{eq:dual-bound} suffices.

\section{The spectrum and the exact mixing norm}\label{sec:spectrum}

We use the right inverse in Lemma~\ref{lem:energy} to determine the transfer norm and construct eigenvectors. The energy scaling fixes the norm of each iterate, and the transfer kernel generates infinite-dimensional eigenspaces throughout the open spectral disk. We first extend pullback to the completed Sobolev space.

\begin{lemma}\label{lem:U-extension}
The pullback $U$ extends boundedly to $W^{1,n}(M)$ with the norm \eqref{eq:norm-equivalence}. On this space,
\begin{equation}\label{eq:U-extension}
 E(Uv)=r^{-1}E(v),\quad \mu(Uv)=\mu(v),\quad
 \Lop Uv=v,\quad \Tr(Uv)=U(\Tr v).
\end{equation}
Both $U$ and $\Lop$ preserve $\Hcal_0$.
\end{lemma}

\begin{proof}
For $v\in\Dcal$, the energy identity and invariance of $\mu$ give
\[
 \norm{Uv}_{\mu,E}=|\mu(v)|+r^{-1}E(v).
\]
Density gives the extension and the first three identities. To prove the trace identity, approximate $v$ in $W^{1,n}$ by smooth functions and use the continuity of the trace and the isometry of $U$ on $L^1(\mu)$. Invariance under $\Lop$ follows from \eqref{eq:trace-intertwine} and stationarity of $\mu$.
\end{proof}

We use the Fredholm essential spectrum
\[
 \spec_{\mathrm{ess}}(T)=\{\lambda\in\C:T-\lambda I\text{ is not Fredholm}\},
\]
where a Fredholm operator has closed range, finite-dimensional kernel, and finite-dimensional cokernel.

\begin{proof}[Proof of Theorem~\ref{thm:spectrum-intro}]
On $\Hcal_0$, Lemmas~\ref{lem:energy} and \ref{lem:U-extension} give $\norm{\Lop^k}\le r^k$. If $v\in\Hcal_0$ is nonzero, then
\[
 \frac{E(\Lop^kU^kv)}{E(U^kv)}
 =\frac{E(v)}{r^{-k}E(v)}=r^k.
\]
Thus the operator norm is exactly $r^k$.

The kernel of $\Lop$ in $\Hcal_0$ is infinite-dimensional. To see this, choose an evenly covered open set $V$ off the branch values, and let $W$ be one of its $d$ inverse sheets. Put
\[
 Q=I-U\Lop.
\]
Then $\Lop Q=0$, $QU=0$, and $\mu(Qa)=0$. The map $a\mapsto Qa$ is injective on $C_c^\infty(W)$. Indeed, if $Qa=0$, then $a=U\Lop a$ is constant on each fiber. Every fiber meeting the support of $a$ has a point in another sheet, where $a=0$, so $a=0$ everywhere. Thus the image of this infinite-dimensional space is contained in $\ker(\Lop|_{\Hcal_0})$. Fix any nonzero $w$ in that kernel.

For $\lambda\in\C$ with $|\lambda|<r$, the series
\begin{equation}\label{eq:eigenvector-series}
 v_\lambda=\sum_{j=0}^{\infty}\lambda^jU^jw
\end{equation}
converges in $\Hcal_0$, since
\[
 E(\lambda^jU^jw)=(|\lambda|/r)^jE(w).
\]
The identity $\Lop w=0$ gives $\Lop v_\lambda=\lambda v_\lambda$, whereas $Qv_\lambda=w$. The linear map $w\mapsto v_\lambda$ is therefore injective on $\ker(\Lop|_{\Hcal_0})$, so every $|\lambda|<r$ has an infinite-dimensional eigenspace. The norm bound gives the reverse spectral inclusion in the closed disk, and closedness of the spectrum gives equality. On the open disk, $\Lop-\lambda I$ is not Fredholm because its kernel is infinite-dimensional. The set of Fredholm operators is open in the operator norm, so the boundary circle also belongs to the essential spectrum. Outside the closed disk the operator is invertible. This proves the essential-spectrum assertion.

Finally,
\[
 W^{1,n}(M)=\C\mathbf1\oplus\Hcal_0
\]
is a topological direct sum invariant under $\Lop$. Its action on the first summand is the identity; on the second, the spectrum lies in $|z|\le r<1$. This proves the asserted spectrum and simplicity of the eigenvalue $1$.
\end{proof}

In particular, $I-\Lop$ is invertible on $\Hcal_0$, and
\begin{equation}\label{eq:resolvent}
 (I-\Lop)^{-1}v=\sum_{j=0}^{\infty}\Lop^jv,
 \qquad
 E\big((I-\Lop)^{-1}v\big)\le\frac{E(v)}{1-r}.
\end{equation}
The series converges in the full Sobolev topology.

\begin{proof}[Proof of Theorem~\ref{thm:main}]
Proposition~\ref{prop:trace-extension} supplies the trace, and Proposition~\ref{prop:prefactor-bound} gives $A_f<\infty$. For $A_f>0$, note that if every smooth function were constant $\mu$-almost everywhere, then $\supp\mu$ would be a single point, contradicting non-atomicity of $\mu$. A smooth function that is nonconstant $\mu$-almost everywhere is nonconstant, so it has positive energy and contributes a positive quotient.

For $\varphi\in\Dcal$, stationarity and the energy contraction give
\[
 \norm{\Lop^k\varphi-\mu(\varphi)}_{L^1(\mu)}
 \le A_fE(\Lop^k\varphi)
 \le A_fr^kE(\varphi).
\]
For the reverse inequality, fix $v\in\Dcal$ with $E(v)>0$ and set $\varphi_k=r^kU^kv$. Then
\[
 E(\varphi_k)=E(v),\qquad
 \Lop^k\varphi_k-\mu(\varphi_k)=r^k(v-\mu(v)).
\]
Taking the supremum over $v$ proves \eqref{eq:sharp-mixing}.
\end{proof}

The same upper estimate holds for arbitrary Sobolev observables using $\Tr\varphi$. Equation~\eqref{eq:duality} then proves \eqref{eq:intro-correlation}, with the corresponding trace interpretation. In particular, taking the supremum over $\norm{\psi}_\infty\le1$ in the correlation recovers the $L^1$ norm in \eqref{eq:sharp-mixing}.

\begin{remark}\label{rem:L2-no-gap}
On $L^2_0(\mu)$, pullback is an isometry and $\Lop U=I$. Consequently,
\[
 \norm{\Lop^k\colon L^2_0(\mu)\to L^2_0(\mu)}=1.
\]
The factor $r$ in the Sobolev norm comes from the energy scaling of pullback. Thus the norm depends on the observable space, even though each fixed centered $L^2$ observable has transfer iterates converging strongly to zero, as proved below.
\end{remark}

\section{Statistical consequences}\label{sec:statistics}

The $L^1$ estimate controls correlations against every bounded observable. We first pass from critical energy to H\"older regularity and then deduce exactness and a limit theorem. These arguments apply to uniformly quasiregular maps in every real dimension $n\ge2$. For rational maps, the mixing and central limit statements belong to the established Sobolev theory \cite{DNS07,DS06}. The paper \cite{DNS07} also gives a convergence rate in the central limit theorem and a local central limit theorem for more regular observables under its non-coboundary hypotheses.

\subsection{H\"older observables and multiple correlations}

For $0<\alpha\le1$, write
\[
 [\varphi]_\alpha=\sup_{x\ne y}
 \frac{|\varphi(x)-\varphi(y)|}{\dist(x,y)^\alpha},
\]
where the distance is that of the smooth background metric.

\begin{theorem}\label{thm:holder}
For $0<\alpha\le1$ there is a constant $H_\alpha$, depending only on the background manifold, $\alpha$, and the ellipticity bound in \eqref{eq:ellipticity}, such that
\begin{equation}\label{eq:holder-decay}
 \norm{\Lop^k\varphi-\mu(\varphi)}_{L^1(\mu)}
 \le H_\alpha(1+C_f)[\varphi]_\alpha r^{\alpha k},
 \qquad \varphi\in C^\alpha(M),\quad k\ge0.
\end{equation}
The corresponding correlation bound follows by multiplying the right side by $\norm{\psi}_{L^\infty(\mu)}$. The exponent $\alpha/n$ is optimal in general, as shown by the expanding torus examples in Section~\ref{subsec:tori}.
\end{theorem}

\begin{proof}
Smoothing in a fixed finite atlas gives, for $0<\varepsilon\le\varepsilon_0$, a smooth function $\varphi_\varepsilon$ with
\[
 \norm{\varphi-\varphi_\varepsilon}_\infty
 \le C_0[\varphi]_\alpha\varepsilon^\alpha,
 \qquad
 E(\varphi_\varepsilon)
 \le C_1[\varphi]_\alpha\varepsilon^{\alpha-1}.
\]
The smoothing can be chosen to preserve constants. The second bound uses \eqref{eq:ellipticity}, so its constant includes the metric comparison. Uniform contraction of $\Lop$ and Theorem~\ref{thm:main} give
\[
 \norm{\Lop^k\varphi-\mu(\varphi)}_1
 \le2C_0[\varphi]_\alpha\varepsilon^\alpha
 +C_fC_1[\varphi]_\alpha r^k\varepsilon^{\alpha-1}.
\]
Choose $\varepsilon=\varepsilon_0r^k$ and absorb the fixed powers of $\varepsilon_0$ into $H_\alpha$.
\end{proof}

Multiple correlations can be estimated without differentiating products of iterated observables. Only the first observable in each remaining time block needs the two-point estimate.

\begin{corollary}\label{cor:multiple}
Let $0=t_0<t_1<\cdots<t_s$ be integers, let $\varphi_0,\ldots,\varphi_{s-1}\in\Dcal$, and let $\varphi_s\in L^\infty(\mu)$. The functions in $\Dcal$ are continuous on a compact manifold, so all the norms below are finite. Then
\begin{align}
 &\abs{\int_M\prod_{j=0}^s(\varphi_j\circ f^{t_j})\,\dd\mu
                  -\prod_{j=0}^s\mu(\varphi_j)}\notag\\
 &\hspace{8mm}\le A_f\sum_{j=0}^{s-1}
 r^{t_{j+1}-t_j}E(\varphi_j)
 \prod_{\substack{0\le\ell\le s\\\ell\ne j}}
 \norm{\varphi_\ell}_{L^\infty(\mu)}.\label{eq:multiple}
\end{align}
For $C^\alpha$ observables, each factor $A_fr^{t_{j+1}-t_j}E(\varphi_j)$ can instead be replaced by
\[
 H_\alpha(1+C_f)r^{\alpha(t_{j+1}-t_j)}[\varphi_j]_\alpha.
\]
\end{corollary}

\begin{proof}
The product after the first observable has the form $\Psi\circ f^{t_1}$, where
\[
 \Psi=\prod_{j=1}^s\varphi_j\circ f^{t_j-t_1},
 \qquad\norm{\Psi}_\infty\le\prod_{j=1}^s\norm{\varphi_j}_\infty.
\]
Apply \eqref{eq:intro-correlation} to $\varphi_0$ and $\Psi$. The remaining integral is the same problem with its first observable removed and all times shifted by $t_1$. Induction, using $|\mu(\varphi_j)|\le\norm{\varphi_j}_\infty$, proves \eqref{eq:multiple}. The H\"older version uses Theorem~\ref{thm:holder} at each step.
\end{proof}

\subsection{Exactness}

Let $\Bcal$ be the completed Borel $\sigma$-algebra of $(M,\mu)$. Exactness means that the tail $\sigma$-algebra $\bigcap_{k\ge0}f^{-k}\Bcal$ is trivial modulo $\mu$.

\begin{theorem}\label{thm:exactness}
For every $v\in L^1(\mu)$,
\begin{equation}\label{eq:L1-exact}
 \norm{\Lop^kv-\mu(v)}_{L^1(\mu)}\longrightarrow0.
\end{equation}
The system $(f,\mu)$ is exact and mixing of all orders. For every $1\le p<\infty$, $\Lop^kv\to\mu(v)$ in $L^p(\mu)$ for all $v\in L^p(\mu)$.
\end{theorem}

\begin{proof}
Smooth functions are dense in $L^1(\mu)$ because $\mu$ is a finite Borel measure on a compact manifold. If $v_j$ is such an approximation, contraction gives
\[
 \norm{\Lop^kv-\mu(v)}_1
 \le2\norm{v-v_j}_1+\norm{\Lop^kv_j-\mu(v_j)}_1.
\]
First let $k\to\infty$ and then $j\to\infty$. Theorem~\ref{thm:main} proves \eqref{eq:L1-exact}.

For completeness, the implication to exactness can be seen directly from the transfer identity. With $\Bcal_k=f^{-k}\Bcal$, conditional expectation with respect to $\mu$ satisfies
\begin{equation}\label{eq:conditional}
 \E_\mu[v\mid\Bcal_k]=U^k\Lop^kv.
\end{equation}
Indeed, testing against bounded $\Bcal_k$-measurable functions reduces to testing against functions of the form $U^k\psi$, where \eqref{eq:conditional} follows from duality and invariance. If $A$ belongs to every $\Bcal_k$, then
\[
 \norm{\mathbf1_A-\mu(A)}_1
 =\norm{U^k(\Lop^k\mathbf1_A-\mu(A))}_1
 =\norm{\Lop^k\mathbf1_A-\mu(A)}_1\longrightarrow0.
\]
Thus $\mu(A)$ is zero or one. This is the transfer-operator characterization of exactness associated with Lin's criterion \cite{Lin71}.

To obtain mixing of all orders for arbitrary bounded observables, repeat the first-step decomposition in the proof of Corollary~\ref{cor:multiple}. Its first error is bounded by
\[
 \norm{\Lop^{t_1}\varphi_0-\mu(\varphi_0)}_1
 \prod_{j=1}^s\norm{\varphi_j}_\infty,
\]
which tends to zero by \eqref{eq:L1-exact}. Induction applies as every consecutive time gap tends to infinity.

Finally, for bounded $v$ and finite $p$,
\begin{equation}\label{eq:Lp-interpolation}
 \norm{\Lop^kv-\mu(v)}_p^p
 \le(2\norm{v}_\infty)^{p-1}
       \norm{\Lop^kv-\mu(v)}_1.
\end{equation}
Approximation by bounded functions and $L^p$ contraction prove the last assertion.
\end{proof}

\subsection{A central limit theorem}

The summable transfer estimates give Gordin's martingale decomposition \cite{Gordin69}. We use it to prove the central limit theorem and characterize zero variance by $L^2$ coboundaries; Hall and Heyde \cite{HallHeyde80} provide the martingale background.

\begin{theorem}\label{thm:CLT}
Let $\varphi$ be real-valued, with $\mu(\varphi)=0$, and suppose that either $\varphi\in\Dcal$ or $\varphi\in C^\alpha(M)$ for some $0<\alpha\le1$. Then, as random variables on the probability space $(M,\mu)$,
\begin{equation}\label{eq:CLT}
 \frac1{\sqrt N}\sum_{j=0}^{N-1}\varphi\circ f^j
 \ \xrightarrow{\ \mathrm{law}\ }\ \mathcal N(0,\sigma_\varphi^2),
\end{equation}
where
\begin{equation}\label{eq:variance}
 \sigma_\varphi^2=\mu(\varphi^2)
 +2\sum_{j=1}^{\infty}\mu\big(\varphi(\varphi\circ f^j)\big).
\end{equation}
The series is absolutely convergent. Moreover, $\sigma_\varphi^2=0$ if and only if
\[
 \varphi=\chi\circ f-\chi
 \quad\mu\text{-almost everywhere}
\]
for some real $\chi\in L^2(\mu)$. Here $\mathcal N(0,0)$ denotes the point mass at zero.
\end{theorem}

\begin{proof}
Theorems~\ref{thm:main} and \ref{thm:holder} give geometric decay of $\norm{\Lop^j\varphi}_1$. Since $\varphi$ is bounded and centered,
\[
 \norm{\Lop^j\varphi}_2^2
 \le\norm{\varphi}_\infty\norm{\Lop^j\varphi}_1.
\]
It follows that $\sum_{j\ge1}\norm{\Lop^j\varphi}_2<\infty$. Following Gordin \cite{Gordin69}, define
\begin{equation}\label{eq:martingale-decomposition}
 \chi=\sum_{j=1}^{\infty}\Lop^j\varphi\in L^2(\mu),
 \qquad v=\varphi+\chi-U\chi.
\end{equation}
Then $\Lop\chi=\chi-\Lop\varphi$ and $\Lop U=I$, so $\Lop v=0$. In particular, by \eqref{eq:conditional}, the sequence $v\circ f^j$ consists of reverse martingale differences for the decreasing $\sigma$-algebras $\Bcal_j$.

To apply the martingale central limit theorem, we verify the conditional variance and Lindeberg conditions. For each $N$, reverse the first $N$ differences and divide them by $\sqrt N$; this is a martingale difference array with an increasing filtration. Its sum of conditional variances is
\[
 \frac1N\sum_{j=0}^{N-1}\big(\Lop(v^2)\big)\circ f^{j+1}.
\]
Ergodicity, which follows from exactness, implies that this expression converges almost everywhere to $\mu(v^2)$. For every $\varepsilon>0$, the expectation of the conditional Lindeberg sum is
\[
 \int_M v^2\mathbf1_{\{|v|>\varepsilon\sqrt N\}}\,\dd\mu,
\]
which tends to zero since $v\in L^2(\mu)$. Thus that sum tends to zero in probability. The martingale central limit theorem \cite{Brown71} applies and gives a normal limit of variance $\mu(v^2)$.

On summing \eqref{eq:martingale-decomposition},
\begin{equation}\label{eq:sum-decomposition}
 \sum_{j=0}^{N-1}\varphi\circ f^j
 =\sum_{j=0}^{N-1}v\circ f^j+U^N\chi-\chi.
\end{equation}
The last two terms divided by $\sqrt N$ tend to zero in $L^2$, because $U$ is an isometry. This proves \eqref{eq:CLT} with variance $\mu(v^2)$.

The correlation estimates imply absolute convergence in \eqref{eq:variance}. Expanding the variance of the left side of \eqref{eq:sum-decomposition} and dividing by $N$ gives the limit in \eqref{eq:variance}. The martingale differences are pairwise orthogonal, and the remainder in \eqref{eq:sum-decomposition} has bounded $L^2$ norm. The same limit is therefore $\mu(v^2)$.

If the variance vanishes, $v=0$ almost everywhere, and \eqref{eq:martingale-decomposition} gives $\varphi=U\chi-\chi$. Conversely, an $L^2$ coboundary has partial sums of uniformly bounded $L^2$ norm. Their normalized variances tend to zero, so \eqref{eq:variance} vanishes.
\end{proof}

For $\varphi\in\Dcal$ the transfer function in \eqref{eq:martingale-decomposition} also comes from the Sobolev space. Indeed, the series converges there by \eqref{eq:resolvent}, and its trace equals the $L^2$ sum. In particular,
\[
 E(\chi)\le\frac{r}{1-r}E(\varphi),
\]
with $E(\chi)$ understood for this Sobolev representative. For a H\"older observable, the decomposition gives $\chi\in L^2(\mu)$.

\section{The trace constant in families}\label{sec:families}

Continuity, conjugacy invariance, and the atomic obstruction are the three ingredients of the moduli compactness theorem. We develop the measure estimates first, then combine them with the degeneration theory of rational maps. Throughout this section the background manifold and its normalized volume are fixed. For a bounded elliptic structure $G$ and a Borel probability measure $\nu$, set
\begin{equation}\label{eq:measure-trace-constant}
 A_G(\nu)=\sup_{\substack{v\in\Dcal\\E_G(v)>0}}
 \frac{\norm{v-\nu(v)}_{L^1(\nu)}}{E_G(v)}\in[0,\infty].
\end{equation}
Thus $A_f=A_G(\mu_f)$ for the structure fixed earlier. The definition makes sense without a dynamical origin for $\nu$.

\subsection{Stability and continuity for rational maps}

For a finite signed measure $\sigma$ of total mass zero, define the possibly infinite discrepancy
\[
 \norm{\sigma}_{E^*}
 =\sup_{\substack{v\in\Dcal\\E_G(v)\le1}}|\sigma(v)|.
\]
Constants do not affect this supremum. The notation uses the fixed energy $E=E_G$.

\begin{lemma}\label{lem:trace-stability}
If $\mu$ and $\nu$ are probabilities with finite trace constants and $\norm{\mu-\nu}_{E^*}<\infty$, then
\begin{equation}\label{eq:trace-stability}
 |A_G(\mu)-A_G(\nu)|\le2\norm{\mu-\nu}_{E^*}.
\end{equation}
\end{lemma}

\begin{proof}
Take $v\in\Dcal$ with $E(v)\le1$. By the Sobolev chain rule, $E(|v-\mu(v)|)\le1$. Adding and subtracting $\nu(|v-\mu(v)|)$, and using the reverse triangle inequality for the absolute value, gives
\begin{align*}
 &\big|\mu(|v-\mu(v)|)-\nu(|v-\nu(v)|)\big|\\
 &\quad\le |(\mu-\nu)(|v-\mu(v)|)|+|\mu(v)-\nu(v)|
 \le2\norm{\mu-\nu}_{E^*}.
\end{align*}
Taking suprema over the energy unit ball proves the assertion.
\end{proof}

For rational maps we always use $S^2$ with its round metric and standard conformal structure. The space $\mathrm{Rat}_d$ consists of coprime pairs of homogeneous polynomials of degree $d$, modulo a common nonzero scalar, with the induced coefficient topology. In particular, a limit taken inside $\mathrm{Rat}_d$ still has degree $d$.

\begin{proposition}\label{prop:rational-continuity}
For every $d\ge2$, the map $f\mapsto A_f$ is continuous on $\mathrm{Rat}_d$. Consequently, the constants $A_f$ are uniformly bounded on every compact subset of $\mathrm{Rat}_d$.
\end{proposition}

\begin{proof}
Let $f_j\to f$ in $\mathrm{Rat}_d$, put $r=d^{-1/2}$, and write $P_g=b_g/(1-r)$ for any map $g$ in this space. The maps converge smoothly on the sphere, and for each fixed $N$ so do $f_j^N$ and $f^N$. Their one-step Jacobians are uniformly bounded. The estimate \eqref{eq:dual-bound}, with $q=2$, therefore gives $\sup_j P_{f_j}<\infty$.

Set $m_{N,g}=d^{-N}(g^N)^*m$. Lemma~\ref{lem:telescoping} gives
\begin{equation}\label{eq:family-discrepancy}
 \norm{\mu_{f_j}-\mu_f}_{E^*}
 \le (P_{f_j}+P_f)r^N
       +\norm{m_{N,f_j}-m_{N,f}}_{E^*}.
\end{equation}
The last signed measure has density $d^{-N}(J_{f_j^N}-J_{f^N})$ of integral zero. Subtracting $m(v)$ in its pairing with $v$ and applying Sobolev--Poincar\'e yields
\[
 \norm{m_{N,f_j}-m_{N,f}}_{E^*}
 \le S_1d^{-N}\norm{J_{f_j^N}-J_{f^N}}_{L^\infty(m)}.
\]
For fixed $N$, this tends to zero as $j\to\infty$. First choose $N$ large in \eqref{eq:family-discrepancy}, and then choose $j$ large. It follows that $\norm{\mu_{f_j}-\mu_f}_{E^*}\to0$, so Lemma~\ref{lem:trace-stability} proves continuity. The compactness assertion follows.
\end{proof}

The proof gives the explicit finite-iterate estimate
\begin{equation}\label{eq:continuity-modulus}
 |A_g-A_f|\le 2(P_g+P_f)r^N
       +2S_1d^{-N}\norm{J_{g^N}-J_{f^N}}_\infty,
 \qquad f,g\in\mathrm{Rat}_d.
\end{equation}
This estimate controls variation within $\mathrm{Rat}_d$. To obtain compactness in moduli, we next establish invariance under changes of coordinates and examine boundary limits of the equilibrium measures.

\subsection{Quasiconformal conjugacy}

The constant $A_f$ depends on the choice of invariant structure, which has been suppressed until now; in this subsection we write $A_{f,G}$ to display it. Transporting $G$ along a conjugacy carries the critical energy with it, leaving the constant unchanged. Derivatives and Jacobians in the following formulas are taken with respect to the fixed smooth background metric, whose choice is not displayed.

An invariant conformal structure need not be unique. The second part of the next proposition compares the constants obtained from structures chosen separately for $f$ and for $g$.

\begin{proposition}\label{prop:qc-conjugacy}
Let $h\colon M\to M$ be an orientation-preserving quasiconformal homeomorphism and let $g=h\circ f\circ h^{-1}$. Transport $G$ by
\begin{equation}\label{eq:transported-G}
 G^h(hx)=J_h(x)^{2/n}Dh(x)^{-\mathsf T}G(x)Dh(x)^{-1}
 \quad\text{almost everywhere}.
\end{equation}
Then $G^h$ is a bounded elliptic, determinant-one structure invariant under $g$, and
\begin{equation}\label{eq:conjugacy-exact}
 \mu_g=h_*\mu_f,\qquad
 E_{G^h}(v)=E_G(v\circ h),\qquad
 A_{g,G^h}=A_{f,G}.
\end{equation}
If instead $G_f$ is any $f$-invariant structure and $G_g$ is any $g$-invariant structure, not necessarily the transport $G_f^h$, with covector comparison constants $\Lambda_f$ and $\Lambda_g$ as in \eqref{eq:ellipticity}, then
\begin{equation}\label{eq:conjugacy-comparison}
 \frac{A_{f,G_f}}{\Lambda_f\Lambda_g K_O(h^{-1})^{1/n}}
 \le A_{g,G_g}
 \le\Lambda_f\Lambda_g K_O(h)^{1/n}A_{f,G_f},
\end{equation}
where $K_O(h)$ is defined by $\norm{Dh}^n\le K_O(h)J_h$ almost everywhere.
\end{proposition}

\begin{proof}
We use the composition and change-of-variables properties of quasiregular mappings from \cite{Rickman93}. The composition formula for distortion shows that $g$ is uniformly quasiregular. Since $J_h>0$ almost everywhere, \eqref{eq:transported-G} defines a positive symmetric structure of determinant one. Quasiconformality of $h$ and ellipticity of $G$ give bounded ellipticity of $G^h$. The identity
\[
 Dh(x)^{\mathsf T}G^h(hx)Dh(x)=J_h(x)^{2/n}G(x),
\]
together with \eqref{eq:beltrami} and the chain rule, proves invariance under $g$. We give the scalar composition argument needed below. For smooth $v$, the Sobolev chain rule and change of variables give
\begin{equation}\label{eq:qc-composition-energy}
 \begin{split}
 E_0(v\circ h)^n
 &\le \int_M |\dd v|^n(hx)\norm{Dh(x)}^n\,\dd m(x)\\
 &\le K_O(h)\int_M |\dd v|^n(hx)J_h(x)\,\dd m(x)
 =K_O(h)E_0(v)^n,
 \end{split}
\end{equation}
where $E_0$ is the background critical energy. Given $v\in\Dcal$, approximate it both uniformly and in $W^{1,n}$ by smooth $v_j$. Applying \eqref{eq:qc-composition-energy} to differences shows that $v_j\circ h$ is Cauchy in $W^{1,n}$; uniform convergence identifies its limit as $v\circ h$. The same argument for $h^{-1}$ proves that composition by $h$ is a bijection of $\Dcal$ and extends \eqref{eq:qc-composition-energy} to that space.

For smooth $v$, the dual form of the displayed identity for $Dh$ gives
\[
 |\dd(v\circ h)|_G(x)^n
 =|\dd v|_{G^h}(hx)^nJ_h(x).
\]
Integration proves $E_{G^h}(v)=E_G(v\circ h)$. The same approximation and ellipticity pass this equality to every $v\in\Dcal$, proving the energy assertion in \eqref{eq:conjugacy-exact}.

To verify equivariance of the measures, let $\omega=h^*m=J_hm$. Higher integrability gives $J_h\in L^q(m)$ for some $q>1$. For $\psi\in\Dcal$, \eqref{eq:energy-pair} and Sobolev--Poincar\'e give
\begin{align*}
 \left|\frac{(f^k)^*\omega}{d^k}(\psi)-m_k(\psi)\right|
 &=\left|\int_M(J_h-1)\Lop^k\psi\,\dd m\right|\\
 &\le S_{q'}\norm{J_h-1}_{L^q(m)}r^kE_G(\psi).
\end{align*}
Here $m(J_h)=1$ because $h$ has degree one. Lemma~\ref{lem:telescoping} and uniform approximation of continuous functions show that $d^{-k}(f^k)^*\omega\rightharpoonup\mu_f$. Functoriality of pullback gives
\[
 \frac{(g^k)^*m}{d^k}
 =h_*\left(\frac{(f^k)^*\omega}{d^k}\right).
\]
Taking weak limits proves $\mu_g=h_*\mu_f$. Together with the energy equality and the bijection $v\mapsto v\circ h$ of $\Dcal$, this proves equality of the optimal trace constants.

Finally, \eqref{eq:qc-composition-energy} and ellipticity give
\[
 E_{G_f}(v\circ h)
 \le\Lambda_f E_0(v\circ h)
 \le\Lambda_f K_O(h)^{1/n}E_0(v)
 \le\Lambda_f\Lambda_g K_O(h)^{1/n}E_{G_g}(v).
\]
Combine this with measure equivariance and take the defining supremum for $A_{g,G_g}$. Applying the same argument to $h^{-1}$ gives the other inequality.
\end{proof}

\subsection{Condensers and an atomic obstruction}

A condenser records oscillation between two sets and is therefore well suited to the centered trace inequality. For disjoint compact sets $E,F\subset M$, define the condenser energy
\[
 \gamma_G(E,F)=\inf\{E_G(u):u\in\Dcal,\ 0\le u\le1,
                    \ u|_E=1,\ u|_F=0\}.
\]
This is the $n$th root of the corresponding variational capacity; the test functions here are real-valued.

\begin{proposition}\label{prop:condenser}
For every probability $\nu$ with $A_G(\nu)<\infty$ and every pair of disjoint compact sets $E,F$,
\begin{equation}\label{eq:condenser-necessary}
 \min\{\nu(E),\nu(F)\}\le A_G(\nu)\gamma_G(E,F).
\end{equation}
\end{proposition}

\begin{proof}
For an admissible $u$, put $c=\nu(u)\in[0,1]$. Then
\[
 \norm{u-c}_{L^1(\nu)}
 \ge\nu(E)(1-c)+\nu(F)c
 \ge\min\{\nu(E),\nu(F)\}.
\]
Apply the trace inequality and take the infimum over admissible $u$.
\end{proof}

This necessary estimate involves mass on both sides of a separation. It is invariant under simultaneous transport of the measure, structure, and condenser plates in Proposition~\ref{prop:qc-conjugacy}. At the critical Sobolev exponent, cutoffs around a point have arbitrarily small energy. This yields the following obstruction when positive mass remains both at the point and away from it.

\begin{proposition}\label{prop:atomic-obstruction}
Let $\mu_j\rightharpoonup\nu$ be probabilities on the fixed closed $n$-manifold, $n\ge2$. Let $G_j$ have a common ellipticity bound as in \eqref{eq:ellipticity}. If there is a point $x$ with
\[
 0<\nu(\{x\})<1,
\]
then $A_{G_j}(\mu_j)\to\infty$.
\end{proposition}

\begin{proof}
Put $p=\nu(\{x\})$ and choose smooth coordinates $w$ near $x$ with $w(x)=0$. Let $\theta\colon\R\to[0,1]$ be smooth, equal to one on $(-\infty,0]$ and zero on $[1,\infty)$. For sufficiently small $\varepsilon>0$, define
\begin{equation}\label{eq:log-cutoff}
 u_\varepsilon(w)=
 \theta\left(\frac{\log(|w|/\varepsilon^2)}{\log(1/\varepsilon)}\right)
 \quad(w\ne0),\qquad u_\varepsilon(x)=1,
\end{equation}
and extend by zero outside the chart. The function is smooth, equals one when $|w|\le\varepsilon^2$, and vanishes when $|w|\ge\varepsilon$. With $E_0$ the background critical energy, comparison in the coordinate chart and polar integration give
\begin{equation}\label{eq:cutoff-energy}
 E_0(u_\varepsilon)^n
 \le\frac{C}{\log(1/\varepsilon)^n}
        \int_{\varepsilon^2}^{\varepsilon}\frac{\dd s}{s}
 =C\log(1/\varepsilon)^{1-n}\longrightarrow0.
\end{equation}
The common ellipticity bound gives $E_{G_j}(u_\varepsilon)\le\Lambda E_0(u_\varepsilon)$ for all $j$.

The functions $u_\varepsilon$ converge pointwise to $\mathbf1_{\{x\}}$ and lie between zero and one. Dominated convergence therefore gives
\[
 \norm{u_\varepsilon-\nu(u_\varepsilon)}_{L^1(\nu)}
 \longrightarrow 2p(1-p)>0.
\]
For each fixed $\varepsilon$, weak convergence yields $\mu_j(u_\varepsilon)\to\nu(u_\varepsilon)$. Consequently the centered absolute-value functions converge uniformly, and
\[
 \norm{u_\varepsilon-\mu_j(u_\varepsilon)}_{L^1(\mu_j)}
 \longrightarrow
 \norm{u_\varepsilon-\nu(u_\varepsilon)}_{L^1(\nu)}.
\]
It follows that
\[
 \liminf_{j\to\infty}A_{G_j}(\mu_j)
 \ge\frac{\norm{u_\varepsilon-\nu(u_\varepsilon)}_{L^1(\nu)}}
          {\Lambda E_0(u_\varepsilon)}.
\]
Letting $\varepsilon\downarrow0$ proves divergence.
\end{proof}

For a point mass, centering gives $A_G(\delta_x)=0$. The family of conjugates of $z^2$ in Section~\ref{subsec:rational-conjugates} illustrates the corresponding concentration with a constant positive value of $A_f$.

\subsection{Compactness in rational moduli}\label{subsec:moduli}

We now restrict to rational maps of a fixed degree $d\ge2$, with the standard conformal structure on the unit sphere $S^2\subset\R^3$. The atomic obstruction becomes a compactness criterion once the equilibrium measures are normalized by M\"obius transformations. We first recall the normalization in a form that applies to singular measures.

\begin{lemma}\label{lem:barycenter-normalization}
Let $\nu$ be a non-atomic Borel probability measure on $S^2$. There is an orientation-preserving M\"obius transformation $h$ such that
\begin{equation}\label{eq:barycenter-normalization}
 \int_{S^2}x\,\dd(h_*\nu)(x)=0\in\R^3.
\end{equation}
\end{lemma}

\begin{proof}
For a non-atomic probability measure $\nu$ on $S^2$, the conformal barycenter construction of Douady and Earle \cite[Section 11]{DE86} can be described using the potential
\[
 h_\nu(x)=\frac12\int_{S^2}\log\frac{1-|x|^2}{|x-u|^2}\,\dd\nu(u),
 \qquad x\in B^3,
\]
where $B^3$ is the open unit ball. Let $\xi_\nu$ be the gradient of $h_\nu$ in the hyperbolic metric. This vector field has a unique zero $B(\nu)\in B^3$, and the barycenter satisfies
\[
 B(h_*\nu)=h(B(\nu)).
\]
Here a M\"obius transformation $h$ of $S^2$ is extended to its hyperbolic isometry of $B^3$. See also \cite[Section 2]{Pet11}.

The vanishing criterion follows from this description. Differentiating under the integral at the origin, the first term contributes nothing and $\nabla_x\log|x-u|^2$ equals $-2u$ there, since $|u|=1$. Hence the Euclidean gradient of $h_\nu$ at the origin is $\int_{S^2}u\,\dd\nu(u)$. The hyperbolic gradient at the origin is a positive multiple of the Euclidean one, so $\xi_\nu(0)$ vanishes precisely when $\int_{S^2}x\,\dd\nu(x)=0$, and uniqueness of the zero gives
\begin{equation}\label{eq:barycenter-vanishing}
 B(\nu)=0\ \Longleftrightarrow\ \int_{S^2}x\,\dd\nu(x)=0.
\end{equation}
Choose an orientation-preserving $h$ taking $B(\nu)$ to the origin; equivariance together with \eqref{eq:barycenter-vanishing} gives \eqref{eq:barycenter-normalization}.
\end{proof}

The degeneration theorem of DeMarco and Faber \cite[Theorem A]{DF14} applies to any sequence $g_j\in\mathrm{Rat}_d$ leaving every compact subset of $\mathrm{Rat}_d$: if $\mu_{g_j}\rightharpoonup\nu$, then $\nu$ is purely atomic. We combine this sequence theorem with barycenter normalization and Proposition~\ref{prop:atomic-obstruction}.

\begin{proof}[Proof of Theorem~\ref{thm:moduli-properness}]
Proposition~\ref{prop:qc-conjugacy} gives
\begin{equation}\label{eq:mobius-moduli-invariance}
 \mu_{h\circ f\circ h^{-1}}=h_*\mu_f,
 \qquad A_{h\circ f\circ h^{-1}}=A_f
\end{equation}
for every orientation-preserving M\"obius transformation $h$: the transported standard conformal structure is again standard. Thus $A_f$ is constant on conjugacy classes. Proposition~\ref{prop:rational-continuity} and the definition of the quotient topology show that the induced function on $\mathcal M_d$ is continuous.

Fix $C>0$ and consider the normalized sublevel set
\begin{equation}\label{eq:normalized-sublevel}
 \mathcal N_{d,C}
 =\left\{g\in\mathrm{Rat}_d:A_g\le C,
             \ \int_{S^2}x\,\dd\mu_g(x)=0\right\}.
\end{equation}
We prove that $\mathcal N_{d,C}$ is compact in $\mathrm{Rat}_d$. Let $(g_j)$ be a sequence in this set. The projective coefficient compactification $\overline{\mathrm{Rat}_d}=\mathbb P^{2d+1}$ is compact. Passing to a subsequence, we may therefore assume that $g_j\to g_\infty$ in this compactification and that $\mu_{g_j}\rightharpoonup\nu$ as probabilities on $S^2$.

Suppose that $g_\infty\notin\mathrm{Rat}_d$. Then $(g_j)$ leaves every compact subset of $\mathrm{Rat}_d$. By the degeneration theorem of DeMarco and Faber \cite[Theorem A]{DF14}, $\nu$ is a countable sum of atoms. The coordinate functions on $S^2$ are continuous, so weak convergence and \eqref{eq:normalized-sublevel} give
\[
 \int_{S^2}x\,\dd\nu(x)=0.
\]
A point mass $\delta_p$, with $p\in S^2$, has Euclidean barycenter $p\ne0$. Hence $\nu$ is not a point mass. As a purely atomic probability, it has an atom of mass strictly between zero and one. Proposition~\ref{prop:atomic-obstruction}, applied with the fixed standard structure, implies $A_{g_j}\to\infty$, contrary to $A_{g_j}\le C$.

It follows that $g_\infty\in\mathrm{Rat}_d$. Proposition~\ref{prop:rational-continuity} gives $A_{g_\infty}\le C$. The dual energy convergence established in its proof applies to the smooth coordinate functions, giving
\[
 \int_{S^2}x\,\dd\mu_{g_\infty}(x)=0,
\]
so $g_\infty\in\mathcal N_{d,C}$. Every sequence in $\mathcal N_{d,C}$ therefore has a subsequence converging in that set. Since coefficient space is metrizable, $\mathcal N_{d,C}$ is compact.

Every equilibrium measure $\mu_f$ is non-atomic. By Lemma~\ref{lem:barycenter-normalization} and \eqref{eq:mobius-moduli-invariance}, each conjugacy class in $\mathcal K_{d,C}$ has a representative in $\mathcal N_{d,C}$. Conversely, every element of $\mathcal N_{d,C}$ represents a class in $\mathcal K_{d,C}$. The quotient map therefore sends the compact set $\mathcal N_{d,C}$ onto $\mathcal K_{d,C}$, proving compactness of the latter.

If $A_f$ is uniformly bounded on $\mathcal F$, its image in moduli lies in one of these compact sublevel sets, which is closed by continuity of $A_f$. Its closure is consequently compact. Conversely, continuity bounds $A_f$ on the compact closure of any relatively compact image. This proves \eqref{eq:moduli-uniformity}. Finally, the inverse image of any compact subset of $\R$ is closed and lies in a compact sublevel set, proving properness.
\end{proof}

The compactness theorem also gives an extremal consequence for the exact mixing norm.

\begin{corollary}\label{cor:moduli-minimum}
For every $d\ge2$, there is a rational map $f_d$ of degree $d$ such that
\[
 0<A_{f_d}=\min_{f\in\mathrm{Rat}_d}A_f.
\]
The set of minimizing conjugacy classes is compact in $\mathcal M_d$. These classes minimize the centered transfer norm in \eqref{eq:sharp-mixing} for every fixed $k\ge0$ among all rational maps of degree $d$.
\end{corollary}

\begin{proof}
Choose $f_0\in\mathrm{Rat}_d$. The sublevel set $\mathcal K_{d,A_{f_0}}$ is nonempty and compact. The continuous function $A_f$ attains its minimum there, and that minimum is global. It is positive by Theorem~\ref{thm:main}. The minimizing classes form a closed subset of this compact sublevel set. Finally, \eqref{eq:sharp-mixing} differs from $A_f$ by the same positive factor $d^{-k/2}$ for every map of degree $d$.
\end{proof}

Corollary~\ref{cor:moduli-minimum} leads to the question of identifying the minimizing classes. In particular, can a minimizing class be represented by a Latt\`es map? Zdunik's characterization of Latt\`es maps by absolute continuity of the maximal entropy measure \cite{Zdu90} makes these natural test cases, with \eqref{eq:density-bound} providing an upper estimate for their trace constants.

The trace and conjugacy arguments above hold in all dimensions. Extending the moduli compactness conclusion would require a compactification of conjugacy classes of pairs $(f,G)$, normalization retaining uniform ellipticity, and control of boundary equilibrium measures. In the rational case, these ingredients are supplied by the coefficient compactification, conformal barycenters, and DeMarco--Faber's theorem. Explicit higher-dimensional uniformly quasiregular dynamics on which such a theory would act is available: see, for instance, Mayer's Latt\`es-type constructions \cite{May97}, the Cantor Julia sets of Fletcher--Wu \cite{FW15}, Fletcher--Stoertz \cite{FS22} and Fletcher--Stoertz--Vellis \cite{FSV} in three dimensions, and the constructions of Iwaniec-Martin \cite{IM96} and Nicks \cite{Nic26}, which produce Cantor Julia sets in every dimension $n\ge3$.

\section{Examples and the failure of uniform constants}\label{sec:examples}

The constant-Jacobian case shows when the trace inequality reduces to Sobolev--Poincar\'e, and expanding tori then make the sharp rates explicit through Fourier modes. Rational families afterwards illustrate two different behaviors of the trace constant: invariance under conformal changes of coordinates, and divergence under degeneration.

\subsection{Constant Jacobian and expanding tori}\label{subsec:tori}

If $J_f=d$ almost everywhere, then $b_f=0$. Lemma~\ref{lem:telescoping} gives $\mu=m$, and $A_f\le2S_1$. Thus the trace inequality reduces to Sobolev--Poincar\'e in this case. This hypothesis cannot hold for a rational map of degree at least two on $S^2$: such a map has critical points, where its continuous spherical Jacobian vanishes.

For an integer $\ell\ge2$, consider
\[
 f(x)=\ell x\pmod{\mathbb Z^n},\qquad x\in\T^n.
\]
This is a conformal covering of degree $d=\ell^n$, with standard $G$, constant Jacobian, and equilibrium measure equal to Haar probability. If $e_\nu(x)=\exp(2\pi i\nu\cdot x)$, then
\[
 \Lop e_\nu=
 \begin{cases}
 e_{\nu/\ell},&\nu\in\ell\mathbb Z^n,\\
 0,&\nu\notin\ell\mathbb Z^n.
 \end{cases}
\]
In particular, for $0<\alpha\le1$ the functions
\[
 \varphi_k(x)=\ell^{-\alpha k}e_{\ell^k e_1}(x)
\]
have uniformly bounded $C^\alpha$ seminorms, mean zero, and
\[
 \norm{\Lop^k\varphi_k}_{L^1(m)}
 =\ell^{-\alpha k}=d^{-\alpha k/n}.
\]
This proves that the H\"older exponent in Theorem~\ref{thm:holder} is optimal in general. Here the test functions vary with the iterate, as permitted by the norm in Theorem~\ref{thm:holder}.

\subsection{Rational maps and conformal changes of coordinates}\label{subsec:rational-conjugates}

On $S^2$, the standard conformal structure gives the normalized Dirichlet energy in \eqref{eq:rational-energy}. Proposition~\ref{prop:qc-conjugacy} then specializes to exact invariance under a M\"obius change of coordinates, since the transported standard structure remains standard. If $h$ is a M\"obius transformation and $f^h=h\circ f\circ h^{-1}$, then $\mu_{f^h}=h_*\mu_f$ and
\[
 E(\varphi\circ h)=E(\varphi),\qquad A_{f^h}=A_f.
\]
This gives a family with concentrating equilibrium measures and a fixed positive trace constant. Namely, $h_a(z)=az$ conjugates $z^2$ to $z^2/a$. As $a\downarrow0$, its equilibrium measure, normalized angular measure on $|z|=a$, converges to $\delta_0$, while $A_{z^2/a}=A_{z^2}$. The family represents a single point of $\mathcal M_2$, illustrating why Theorem~\ref{thm:moduli-properness} is formulated modulo M\"obius conjugacy.

Here centering removes the oscillation of every fixed continuous test function in the point-mass limit. In the degeneration below, mass persists both at an atom and away from it, giving positive centered oscillation for cutoffs of arbitrarily small energy. The two masses are exactly what the condenser estimate \eqref{eq:condenser-necessary} observes.

The discrepancy constant $b_f$ is more sensitive to the choice of background volume. For this last family it tends to infinity. If it remained bounded along a sequence $a\downarrow0$, \eqref{eq:potential-zero} would pass to the limit and give
\[
 |v(0)-m(v)|\le C E(v),\qquad v\in C^\infty(S^2).
\]
Smooth cutoffs equal to one near zero, with supports shrinking to zero and Dirichlet energies tending to zero, contradict this inequality. The cutoffs in \eqref{eq:log-cutoff} have these properties. Thus $b_f$ and the Jacobian norm in \eqref{eq:main-constant} can diverge while $A_f$ stays fixed. Their dependence on background volume explains the difference from the conjugacy-invariant constant $A_f$.

\subsection{A quadratic degeneration}

We now prove the divergence asserted in Theorem~\ref{thm:nonuniform-intro}. Each $f_t$ in \eqref{eq:degenerate-family} has degree two: its numerator and denominator have no common zero for $0<t<1$. As each $f_t$ is rational, it is uniformly $1$-quasiregular on $S^2$. Write $\mu_t=\mu_{f_t}$ and $\Lop_t$ for its normalized transfer operator.

The following limiting measure is a special case of the atomic boundary measures in \cite{DeM05}. The calculation also supplies the convergence needed for the trace argument.

\begin{lemma}\label{lem:atomic-limit}
As $t\downarrow0$,
\begin{equation}\label{eq:atomic-limit}
 \mu_t\ \rightharpoonup\
 \nu:=\sum_{j=0}^{\infty}2^{-j-1}\delta_{2^j}.
\end{equation}
\end{lemma}

\begin{proof}
The equation $f_t(z)=y$ is
\[
 z(z-1)-y\big(2(z-1)+t\big)=0.
\]
At $t=0$ its roots, with multiplicity, are $1$ and $2y$. This statement includes $y=\infty$ when read homogeneously. Indeed, for $y=[Y:W]$, the limiting homogeneous polynomial is
\[
 (Z-S)(WZ-2YS),
\]
which is not identically zero for any $[Y:W]\in\mathbb P^1$. Continuity of roots as unordered pairs on the sphere therefore gives, for each $\varphi\in C(S^2)$,
\begin{equation}\label{eq:Q-limit}
 \Lop_t\varphi\longrightarrow Q\varphi
 \quad\text{uniformly on }S^2,
 \qquad Q\varphi(y)=\tfrac12\varphi(1)+\tfrac12\varphi(2y).
\end{equation}
To see uniformity explicitly, a sequence of points where it failed would have a convergent subsequence $y_t\to y$. The homogeneous polynomials and their unordered root pairs would converge at that subsequence to the nonzero limiting polynomial above, contradicting continuity of $\varphi$. Double roots cause no difficulty because multiplicities are retained.

Take any weakly convergent subsequence of the probabilities $\mu_t$ and call its limit $\nu_0$. Stationarity and \eqref{eq:Q-limit} give
\[
 \nu_0(\varphi)=\nu_0(Q\varphi),
 \quad\text{or equivalently}\quad
 \nu_0=\tfrac12\delta_1+\tfrac12h_*\nu_0,
 \qquad h(y)=2y.
\]
Iterating the last identity $N$ times gives
\[
 \nu_0=\sum_{j=0}^{N-1}2^{-j-1}\delta_{2^j}
       +2^{-N}(h^N)_*\nu_0.
\]
The last term has mass $2^{-N}$, so the stationary probability is uniquely the measure in \eqref{eq:atomic-limit}. Every subsequential limit is therefore $\nu$, proving convergence of the whole family.
\end{proof}

\begin{proof}[Proof of Theorem~\ref{thm:nonuniform-intro}]
The measure in \eqref{eq:atomic-limit} has an atom of mass $1/4$ at $2$. The background energy is fixed along the family, so Proposition~\ref{prop:atomic-obstruction} gives $A_{f_t}\to\infty$. More explicitly, take the cutoffs in \eqref{eq:log-cutoff} centered at $2$. For all sufficiently small $\varepsilon$, their supports contain no other atom of $\nu$, and hence
\[
 \nu(u_\varepsilon)=\tfrac14,
 \qquad
 \norm{u_\varepsilon-\nu(u_\varepsilon)}_{L^1(\nu)}=\tfrac38.
\]
For each fixed $\varepsilon$, weak convergence gives
\[
 \liminf_{t\downarrow0}A_{f_t}\ge\frac{3}{8E(u_\varepsilon)}.
\]
Now \eqref{eq:cutoff-energy} lets $\varepsilon\downarrow0$. By \eqref{eq:sharp-mixing}, the centered transfer norms diverge at every fixed $k$. Here $n=2$, $K=1$, and $d=2$ throughout, proving the final assertion.
\end{proof}

\begin{corollary}\label{cor:jacobian-divergence}
For the family \eqref{eq:degenerate-family},
\[
 \norm{J_{f_t}/2-1}_{L^q(m)}\longrightarrow\infty
 \qquad(t\downarrow0)
\]
for every fixed $q>1$.
\end{corollary}

\begin{proof}
Each norm is finite for $t>0$, since a rational map is smooth as a map of the sphere. In \eqref{eq:main-constant} the constants $S_{q'}$ and $r=2^{-1/2}$ are fixed along the family, so a bounded subsequence of Jacobian norms would bound $A_{f_t}$ along that subsequence, contradicting Theorem~\ref{thm:nonuniform-intro}.
\end{proof}

The example separates the fixed, degree-determined exponential rate from the map-dependent size of its prefactor. Since $A_f$ is bounded on compact subsets of $\mathcal M_2$, the divergence $A_{f_t}\to\infty$ also shows that $[f_t]$ escapes every compact subset of $\mathcal M_2$.

\subsection*{Acknowledgements}
ChatGPT and Claude assisted with proof development, literature searches, mathematical exposition, editing, and LaTeX preparation. The authors are responsible for the mathematical content and its attribution.


\end{document}